\documentclass[10pt,a4paper]{amsart}
\usepackage[latin1]{inputenc}
\usepackage[T1]{fontenc}
\usepackage[english]{babel}

\usepackage{amsmath,amssymb,amsthm,amsfonts}

\theoremstyle{plain}
\newtheorem{theorem}{Theorem}[section]
\newtheorem{corollary}[theorem]{Corollary}
\newtheorem{prop}[theorem]{Proposition}
\newtheorem{lemma}[theorem]{Lemma}

\theoremstyle{definition}
\newtheorem{remark}[theorem]{Remark}

\newtheorem{example}[theorem]{Example}
\newtheorem{examples}[theorem]{Examples}
\newtheorem{problem}[theorem]{Problem}

\DeclareMathOperator{\e}{e}
\newcommand{\C}{\mathbb{C}}

\newcommand{\R}{\mathbb{R}}
\newcommand{\K}{\mathbb{K}}
\newcommand{\N}{\mathbb{N}}

\newcommand{\eps}{\varepsilon}

 \DeclareMathOperator{\Id}{Id}
 \DeclareMathOperator{\linspan}{span}

\newcommand{\conv}{\mathrm{conv}}
\newcommand{\cconv}{\overline{\mathrm{conv}}}

 \renewcommand{\leq}{\leqslant}
\renewcommand{\geq}{\geqslant}

\begin{document}

\title[Numerical index of absolute sums of Banach spaces]
{Numerical index of absolute sums of Banach spaces}

\author[Mart\'{\i}n]{Miguel Mart\'{\i}n}
\author[Mer\'{\i}]{Javier Mer\'{\i}}

\address[Mart\'{\i}n \& Mer\'{\i}]{Departamento de An\'{a}lisis Matematico\\
Facultad de Ciencias\\
Universidad de Granada\\
18071 Granada, Spain}

\email{\texttt{mmartins@ugr.es} \qquad \texttt{jmeri@ugr.es}}

\author[Popov]{Mikhail Popov}

\address[Popov]{Department of Mathematics\\
Chernivtsi National University\\ str. Kotsjubyn'skogo 2,
Chernivtsi, 58012 Ukraine}

\email{\texttt{misham.popov@gmail.com}}

\author[Randrianantoanina]{Beata Randrianantoanina}

\address[Randrianantoanina]{Department of Mathematics
\\ Miami University \\ Oxford, OH 45056, USA}

 \email{\texttt{randrib@muohio.edu}}

\date{March 16th, 2010}

\thanks{First and second authors partially supported by Spanish
MICINN and FEDER project no.\ MTM2009-07498 and Junta de Andaluc\'{\i}a
and FEDER grants G09-FQM-185, P06-FQM-01438 and P09-FQM-4911.
Third author supported by Junta de Andaluc\'{\i}a and FEDER grant
P06-FQM-01438 and by Ukr.\ Derzh.\ Tema N 0103Y001103. Fourth author participant, NSF Workshop
in Linear Analysis and Probability, Texas A\&M University}



\begin{abstract}
We study the numerical index of absolute sums of Banach spaces,
giving general conditions which imply that the numerical index of
the sum is less or equal than the infimum of the numerical indices
of the summands and we provide some examples where the equality
holds covering the already known case of $c_0$-, $\ell_1$- and
$\ell_\infty$-sums and giving as a new result the case of $E$-sums
where $E$ has the RNP and $n(E)=1$ (in particular for
finite-dimensional $E$ with $n(E)=1$). We also show that the
numerical index of a Banach space $Z$ which contains a dense
increasing union of one-complemented subspaces is greater or equal
than the limit superior of the numerical indices of those
subspaces. Using  these results, we give a detailed short proof of
the already known fact that the numerical indices of all
infinite-dimensional $L_p(\mu)$-spaces coincide.
\end{abstract}

\maketitle

\section{Introduction}

Given a Banach space $X$, we write $B_X$, $S_X$ and $X^*$ to
denote its closed unit ball, its unit sphere and its topological
dual and define
$$
\Pi(X):=\bigl\{(x,x^*)\in S_X \times S_{X^*}\ :\  x^*(x)=1 \bigr\},
$$
and denote the Banach algebra of all (bounded linear) operators on
$X$  by $L(X)$. For an operator $T\in L(X)$, its \emph{numerical
radius} is defined as
$$
v(T):=\sup\{|x^*(Tx)| \ : \ (x,x^*)\in \Pi(X) \},
$$
which is a seminorm on $L(X)$ smaller than the operator norm. The
\emph{numerical index} of $X$ is the constant given by
\begin{align*}
n(X) & :=\inf\{v(T) \ : \ T\in L(X),\ \|T\|=1\}
= \max \{k\geq 0\ :\ k\,\|T\| \leq v(T)\ \forall T\in L(X)\}.
\end{align*}
The numerical radius of bounded linear operators on Banach
spaces was introduced, independently, by F.~Bauer and  G.~Lumer
in the 1960's extending the Hilbert space case from the 1910's.
The definition of numerical index appeared for the first time
in the 1970 paper \cite{D-Mc-P-W}, where the authors attributed
the authorship of the concept to G.~Lumer. Classical references
here are the monographs by F.~Bonsall and J.~Duncan
\cite{B-D1,B-D2} from the 1970's. The reader will find the
state-of-the-art on numerical indices in the survey paper
\cite{KaMaPa} and references therein. We refer to all these
references for background. Only newer results which are not
covered there will be explicitly referenced in this
introduction.

Let us present here the context necessary for the paper. First,
real and complex Banach spaces do not behave in the same way with
respect to numerical indices. In the real case, all values in
$[0,1]$ are possible for the numerical index. In the complex case,
$1/\e\leq n(X)\leq 1$ and all of these values are possible. There
are some classical Banach spaces for which the numerical index has
been calculated. For instance, the numerical index of $L_1(\mu)$
is $1$, and this property is shared by any of its isometric
preduals. In particular, $n\bigl(C(K)\bigr)=1$ for every compact
$K$. Also, $n(Y)=1$ for every finite-codimensional subspace $Y$ of
$C[0,1]$. If $H$ is a Hilbert space of dimension greater than one
then $n(H)=0$ in the real case and $n(H)=1/2$ in the complex case.
The exact value of the numerical indices of $L_p(\mu)$ spaces is
still unknown when $1<p<\infty$ and $p\neq 2$, but it is is known
\cite{Eddari,Eddari-Khamsi} that all infinite-dimensional
$L_p(\mu)$ spaces have the same numerical index, which coincides
with the infimum of the numerical indices of finite-dimensional
$L_p(\mu)$ spaces, and the result has been extended to
vector-valued $L_p$ spaces \cite{Eddari-Khamsi-Aksoy}. It has been
shown very recently \cite{MarMerPop} that every real $L_p(\mu)$
space has positive numerical index for $p\neq 2$. Some known
results about absolute sums of Banach spaces and about
vector-valued function spaces are the following. The numerical
index of the $c_0$-, $\ell_1$- or $\ell_\infty$-sum of a family of
Banach spaces coincides with the infimum of the numerical indices
of the elements of the family, while the numerical index of the
$\ell_p$-sum is only smaller or equal than the infimum. For a
Banach space $X$, it is known that, among others, the following
spaces have the same numerical index as $X$: $C(K,X)$,
$L_1(\mu,X)$, $L_\infty(\mu,X)$.

Our main goal in this paper is to study the numerical index of
absolute sums of Banach spaces. Given a nonempty set $\Lambda$
and a linear subspace $E$ of $\R^\Lambda$ with absolute norm
(see Section~\ref{sec:absolute-sums} for the exact definition),
we may define the $E$-sum of a family of Banach spaces indexed
in $\Lambda$. We give very general conditions on the space $E$
to assure that the numerical index of an $E$-sum of a family of
Banach spaces is smaller or equal than the infimum of the
numerical indices of the elements of the family. It covers the
already known case of $\ell_p$-sums ($1\leq p \leq \infty$) and
also the case when $E$ is a Banach space with a
one-unconditional basis. On the other hand, we give a condition
on $E$ to get that the numerical index of an $E$-sum of a
family of Banach spaces is equal to the infimum of the
numerical indices of the elements of the family. As a
consequence, we obtain the already known result for $c_0$-,
$\ell_1$- and $\ell_\infty$-sums with a unified approach and
also the case of $E$-sums when $E$ has the Radon-Nikod\'{y}m
property (RNP in short) and $n(E)=1$. In particular, the
numerical index of a finite $E$-sum of Banach spaces is equal
to the minimum of the numerical indices of the summands when
$n(E)=1$.

Besides of the above results, we discuss in
Section~\ref{sec:subspaces-examples} when the numerical index
of a Banach space is smaller than the numerical index of its
one-complemented subspaces, and we give examples showing that
this is not always the case for unconditional subspaces. We
show in Section~\ref{sec:kothe} sufficient conditions on a
K\"{o}the space $E$ to ensure that $n\bigl(E(X)\bigr)\leq n(X)$ for
every Banach space $X$. These conditions cover the already
known cases of $E=L_p(\mu)$ ($1\leq p \leq \infty$) with a
unified approach but they also give the case of an
order-continuous K\"{o}the space $E$. In Section~\ref{sec:union} it
is shown that the numerical index of a Banach space which
contains a dense increasing union of one-complemented subspaces
is greater or equal than the limit superior of the numerical
indices of those subspaces. As a consequence, if a Banach space
has a monotone basis, its numerical index is greater or equal
than the limit superior of the numerical indices of the ranges
of the projections associated to the basis.

Finally, in Section~\ref{sec:Lp} we deduce from the results of the
previous sections the already known result
\cite{Eddari,Eddari-Khamsi,Eddari-Khamsi-Aksoy} that for every
positive measure $\mu$ such that $L_p(\mu)$ is
infinite-dimensional and every Banach space $X$,
$n\bigl(L_p(\mu,X)\bigr)= n\bigl(\ell_p(X)\bigr)= \inf_{m\in\N}
n\bigl(\ell_p^m(X)\bigr)$. In our opinion, the abstract vision we
are developing in this paper allows to understand better the
properties of $L_p$-spaces underlying the proofs: $\ell_p$-sums
are absolute sums, $L_p$-norms are associative, every measure
space can be decomposed into parts of finite measure, every finite
measure algebra is isomorphic to the union of homogeneous measure
algebras (Maharam's theorem) and, finally, the density of simple
functions via the conditional expectation projections.

We recall that given a measure space $(\Omega,\Sigma,\mu)$ and a
Banach space $X$, $L_p(\mu,X)$ denotes the Banach space of
(equivalent classes of) Bochner-measurable functions from $\Omega$
into $X$. Let us observe that we may suppose that the measure
$\mu$ is complete since every positive measure and its completion
provide the same vector-valued $L_p$-spaces. When $\Omega$ has $m$
elements and $\mu$ is the counting measure, we write
$\ell_p^m(X)$. When $\Omega$ is an infinite countable set and
$\mu$ is the counting measure, we write $\ell_p(X)$. We will write
$X\oplus_p Y$ to denote the $\ell_p$-sum of two Banach spaces $X$
and $Y$.

We finish the introduction with the following result from
\cite{B-D1} which allows to calculate numerical radii of operators
using a dense subset of the unit sphere and one supporting
functional for each of the points of this dense subset, and which
we will use along the paper.

\begin{lemma}[\textrm{\cite[Theorem~9.3]{B-D1}}]\label{lemma:Theorem93}
Let $X$ be a Banach space, let $\Gamma$ be subset of $\Pi(X)$ such
that the projection on the first coordinate is dense in $S_X$.
Then
$$
v(T)=\sup\{|x^*(Tx)|\ : \ (x,x^*)\in \Gamma\}
$$
for every $T\in L(X)$.
\end{lemma}

\section{Absolute sums of Banach spaces}\label{sec:absolute-sums}
Let $\Lambda$ be a nonempty set and let $E$ be a linear subspace
of $\R^\Lambda$. An \emph{absolute norm} on $E$ is a complete norm
$\|\cdot\|_E$ satisfying
\begin{itemize}
\item[(a)] Given $(a_\lambda), (b_\lambda)\in \R^\Lambda$ with
    $|a_\lambda| = |b_\lambda|$ for every $\lambda\in
    \Lambda$, if $(a_\lambda)\in E$, then $(b_\lambda)\in E$
    with $\|(a_\lambda)\|_E = \|(b_\lambda)\|_E$.
\item[(b)] For every $\lambda\in \Lambda$,
    $\chi_{\{\lambda\}}\in E$ with
    $\|\chi_{\{\lambda\}}\|_E=1$, where
    $\chi_{\{\lambda\}}$ is the characteristic function of
    the singlet $\{\lambda\}$.
\end{itemize}
The following results can be deduced from the definition above:
\begin{itemize}
\item[(c)] Given $(x_\lambda), (y_\lambda)\in \R^\Lambda$ with
    $|y_\lambda| \leq |x_\lambda|$ for every $\lambda\in
    \Lambda$, if $(x_\lambda)\in E$, then $(y_\lambda)\in E$
    with $\|(y_\lambda)\|_E \leq \|(x_\lambda)\|_E$.
\item[(d)] $\ell_1(\Lambda)\subseteq E \subseteq
    \ell_\infty(\Lambda)$ with contractive inclusions.
\end{itemize}
Observe that $E$ is a Banach lattice in the pointwise order
(actually, $E$ can be viewed as a K\"{o}the space on the measure
space $(\Lambda,\mathcal{P}(\Lambda),\nu)$ where $\nu$ is the
counting measure on $\Lambda$, which is non-necessarily
$\sigma$-finite, see Section~\ref{sec:kothe}). The \emph{K\"{o}the
dual} $E'$ of $E$ is the linear subspace of $\R^\Lambda$
defined by
$$
E'=\left\{(b_\lambda)\in \R^\Lambda\ : \ \|(b_\lambda)\|_{E'}:=
\sup_{(a_\lambda)\in B_E} \sum_{\lambda\in\Lambda}|b_\lambda||a_\lambda| <\infty \right\}.
$$
The norm $\|\cdot\|_{E'}$ on $E'$ is an absolute norm. Every
element $(b_\lambda)\in E'$ defines naturally a continuous
linear functional on $E$ by the formula
$$
(a_\lambda)\longmapsto \sum_{\lambda\in\Lambda}b_\lambda a_\lambda \qquad \bigl((a_\lambda)\in E\bigr),
$$
so we have $E'\subseteq E^*$ and this inclusion is isometric. We
say that $E$ is \emph{order continuous} if $0\leq x_\alpha
\downarrow 0$ and $x_\alpha\in E$ imply that $\lim \|x_\alpha\|=0$
(since $E$ is order complete, this is known to be equivalent to
the fact that $E$ does not contain an isomorphic copy of
$\ell_\infty$, see \cite{LTII}). If $E$ is order continuous, the
set of those functions with finite support is dense in $E$ and the
inclusion $E'\subseteq E^*$ is surjective (this is shown for K\"{o}the
spaces defined on a $\sigma$-finite space by using the
Radon-Nikod\'{y}m theorem; in the case we are studying here, the
measure spaces are not necessarily $\sigma$-finite, but since they
are discrete the proof of the fact that $E'=E^*$ is
straightforward).

Given an arbitrary family $\{X_\lambda\, : \, \lambda\in\Lambda\}$
of Banach spaces, the \emph{$E$-sum} of the family is the space
\begin{align*}
\Bigl[\bigoplus_{\lambda\in \Lambda} X_\lambda\Bigr]_{E} &:=
\Bigl\{(x_\lambda)\,:\, x_\lambda\in X_\lambda\ \forall \lambda\in \Lambda,\ (\|x_\lambda\|)\in E \Bigr\}
\\ & =\Bigl\{(a_\lambda x_\lambda)\,:\, a_\lambda\in \R^+_0,\ x_\lambda\in S_{X_\lambda}\ \forall \lambda\in \Lambda,
\ (a_\lambda)\in E \Bigr\}
\end{align*}
endowed with the complete norm
$\|(x_\lambda)\|=\|(\|x_\lambda\|)\|_E$. We will use the name
\emph{absolute sum} when the space $E$ is clear from the context.
Write $X=\Bigl[\bigoplus_{\lambda\in \Lambda}
X_\lambda\Bigr]_{E}$. For every $\kappa\in\Lambda$, we consider
the natural inclusion $I_\kappa:X_\kappa\longrightarrow X$ given
by $I_\kappa(x)=x\,\chi_{\{\kappa\}}$ for every $x\in X_\kappa$,
which is an isometric embedding, and the natural projection
$P_\kappa:X \longrightarrow X_\kappa$ given by
$P_\kappa\bigl((x_\lambda)\bigr)=x_\kappa$ for every
$(x_\lambda)\in X$, which is contractive. Clearly, $P_\kappa
I_\kappa=\Id_{X_\kappa}$. We write
$$
X'=\Bigl[\bigoplus_{\lambda\in \Lambda} X^*_\lambda\Bigr]_{E'}
$$
and observe that every element in $(x^*_\lambda)\in X'$ defines
naturally a continuous linear functional on $X$ by the formula
$$
(x_\lambda)\longmapsto \sum_{\lambda\in\Lambda}x^*_\lambda(x_\lambda) \qquad \bigl((x_\lambda)\in E\bigr),
$$
so we have $X'\subseteq X^*$ and this inclusion is isometric.

Examples of absolute sums are $c_0$-sums, $\ell_p$-sums for $1\leq
p\leq \infty$, i.e.\ given a nonempty set $\Lambda$, we are
considering $E=c_0(\Lambda)$ or $E=\ell_p(\Lambda)$. More examples
are the absolute sums produced using a Banach space $E$ with a
one-unconditional basis, finite (i.e.\ $E$ is $\R^m$ endowed with
an absolute norm) or infinite (i.e.\ $E$ is a Banach space with an
one-unconditional basis viewed as a linear subspace of $\R^\N$ via
the basis).

Our first main result gives an inequality between the numerical
index of an $E$-sum of Banach spaces and the infimum of the
numerical index of the summands, provided that $E'$ contains
sufficiently many norm-attaining functionals.

\begin{theorem}\label{thm:EMA-infinito}
Let $\Lambda$ be a non-empty set and let $E$ be a linear
subspace of $\R^\Lambda$ endowed with an absolute norm. Suppose
there is a dense subset $A\subseteq S_E$ such that for every
$(a_\lambda)\in A$, there exists $(b_\lambda)\in S_{E'}$
satisfying $\sum_{\lambda\in \Lambda} b_\lambda a_\lambda=1$.
Then, given an arbitrary family $\{X_\lambda\, : \,
\lambda\in\Lambda\}$ of Banach spaces,
$$
n\left(\Bigl[\bigoplus_{\lambda\in \Lambda} X_\lambda\Bigr]_{E}\right) \,
\leq \, \inf\bigl\{n(X_\lambda)\,:\,\lambda\in \Lambda\bigr\}.
$$
\end{theorem}

\begin{proof}
Write $X=\Bigl[\bigoplus_{\lambda\in \Lambda} X_\lambda\Bigr]_{E}$
and $X'=\Bigl[\bigoplus_{\lambda\in \Lambda}
X^*_\lambda\Bigr]_{E'}\subseteq X^*$. Fix $\kappa\in \Lambda$. For
every $S\in L(X_\kappa)$, we define $T\in L(X)$ by $T=I_{\kappa} S
P_\kappa$. It then follows that $\|T\|\leq \|S\|$. Since $S=
P_\kappa T I_\kappa$, $\|S\|\leq \|T\|$ and so $\|T\|=\|S\|$.

We claim that $v(T)\leq v(S)$. Indeed, we consider the set
$\mathcal{A}\subseteq S_X$ given by
\begin{equation*}
\mathcal{A} = \left\{(a_\lambda x_\lambda)\ :\ a_\lambda\in \R^+_0,\,x_\lambda \in S_{X_\lambda}\,\forall \lambda\in\Lambda,\
               (a_\lambda)\in A\right\}
\end{equation*}
and for every $a=(a_\lambda x_\lambda)\in \mathcal{A}$, we write
$$
\Upsilon(a)=(b_\lambda x_\lambda^*)\in S_{X'}\subseteq S_{X^*}
$$
where $x_\lambda^*\in S_{X_\lambda^*}$ satisfies
$x_\lambda^*(x_\lambda)=1$ and $(b_\lambda)\in S_{E'}$ satisfies
$\sum_{\lambda\in \Lambda} b_\lambda a_\lambda=1$. The set
$\mathcal{A}$ is dense in $S_X$ and $[\Upsilon(a)](a)=1$ for every
$a\in\mathcal{A}$. It then follows from
Lemma~\ref{lemma:Theorem93} that
$$
v(T)=\sup\left\{|[\Upsilon(a)](T(a))|\,:\, a\in\mathcal{A} \right\}.
$$
For every $a\in \mathcal{A}$, we have
\begin{align*}
|[\Upsilon(a)](T(a))| = |[\Upsilon(a)](I_\kappa(S(a_\kappa x_\kappa)))|
    =b_\kappa a_\kappa |x_\kappa^*(S(x_\kappa))| \leq |x_\kappa^*(S(x_\kappa))| \leq v(S),
\end{align*}
where the last inequality follows from the fact that
$(x_\kappa,x_\kappa^*)\in \Pi(X_\kappa)$. Taking supremum with
$a\in\mathcal{A}$, we get $v(T)\leq v(S)$ as desired.

Now, we observe that
$$
v(S)\geq v(T)\geq n(X)\|T\|\geq n(X)\|S\|,
$$
and the arbitrariness of $S\in L(X_\kappa)$ gives us that
$n(X_\kappa)\geq n(X)$.
\end{proof}

Let us list here the main consequences of the above result.

Let $E$ be a linear subspace of $\R^\Lambda$ with an absolute
norm. If $E$ is order continuous, the hypotheses of
Theorem~\ref{thm:EMA-infinito} are trivially satisfied (since
$E^*=E'$). Again in this case, $E'$ is a linear subspace of
$\R^\Lambda$ with an absolute norm and the hypotheses of
Theorem~\ref{thm:EMA-infinito} are satisfied for $E'$ thanks to
the Bishop-Phelps theorem (the set of those norm-one elements in
$E'=E^*$ attaining the norm on $E\subset E''$ is dense in the unit
sphere of $E'$). Therefore, the following result follows.

\begin{corollary}\label{cor:leq-Eprima-F}
Let $\Lambda$ be a nonempty set and let $E$ be a linear subspace
of $\R^\Lambda$ with an absolute norm which is order continuous.
Then, given an arbitrary family $\{X_\lambda\, : \,
\lambda\in\Lambda\}$ of Banach spaces,
\begin{align*}
n\left(\Bigl[\bigoplus_{\lambda\in \Lambda}
X_\lambda\Bigr]_{E}\right)   \leq \,
\inf\bigl\{n(X_\lambda)\,:\,\lambda\in \Lambda\bigr\}, \qquad \text{and} \qquad
n\left(\Bigl[\bigoplus_{\lambda\in \Lambda}
X_\lambda\Bigr]_{E'}\right)  \leq \,
\inf\bigl\{n(X_\lambda)\,:\,\lambda\in \Lambda\bigr\}.
\end{align*}
\end{corollary}

The spaces $E=c_0(\Lambda)$ and $E=\ell_p(\Lambda)$ for $1\leq p <
\infty$ are order continuous. For $E=\ell_1(\Lambda)$ we have
$E^*=E'=\ell_\infty(\Lambda)$. Therefore, the following corollary
follows from the above result. It appeared in \cite[Proposition~1
and Remark~2.a]{M-P}.

\begin{corollary}\label{cor:c0-lp-sum-easyinequality}
Let $\Lambda$ be a non-empty set and let $\{X_\lambda\, : \,
\lambda\in\Lambda\}$ be a family of Banach spaces. Let $X$ denote
the $c_0$-sum or $\ell_p$-sum of the family ($1\leq p \leq
\infty$). Then
$$
n(X) \,
\leq \, \inf\bigl\{n(X_\lambda)\,:\,\lambda\in \Lambda\bigr\}.
$$
\end{corollary}

A particular case of the above corollary is the absolute sums
associated to a Banach space with one-unconditional basis (finite
or infinite). Related to infinite bases, let us comment that order
continuous linear subspaces of $\R^\N$ with absolute norm have
one-unconditional basis and, reciprocally, if a Banach space has a
one-unconditional basis it can be viewed (via the basis) as an
order continuous linear subspace of $\R^\N$ with absolute norm.

\begin{corollary}$ $\label{cor:One-unconditonalbasis-easy}
\begin{itemize}
\item[(a)] Let $E$ be $\R^m$ endowed with an absolute norm and
    let $X_1,\ldots,X_m$ be Banach spaces. Then
    $$
    n\left(\bigl[X_1\oplus \cdots \oplus X_m\bigr]_{E} \right)
        \leq \min\bigl\{n(X_1),\ldots,n(X_m)\bigr\}.
    $$
\item[(b)] Let $E$ be a Banach space with a one-unconditional
    (infinite) basis and let $\{X_j\, : \, j\in\N\}$ be a
    sequence of Banach spaces. Then
$$
n\left(\Bigl[\bigoplus_{j\in \N} X_j\Bigr]_{E}\right) \,
\leq \, \inf\bigl\{n(X_j)\,:\,j\in \N\bigr\}.
$$
\end{itemize}
\end{corollary}

Our goal in the rest of the section is to present some cases in
which we may get the reversed inequality to the one given in
Theorem~\ref{thm:EMA-infinito}. When both theorems are applied,
we get an exact formula for the numerical index of some
absolute sums. The more general result we are able to prove is
the following.

\begin{theorem}\label{thm:absolutesum-equality}
Let $\Lambda$ be a non-empty set and let $E$ be a subspace or
$\R^\Lambda$ endowed with an absolute norm. Suppose that there are
a subset $A\subseteq S_E$ with $\cconv(A)=B_E$ and a subset
$B\subseteq S_{E'}$ norming for $E$ such that for every
$(a_\lambda)\in A$ and every $(b_\lambda)\in B$, there is
$\kappa\in \Lambda$ such that
$$
a_\lambda b_\lambda=0 \ \ \ \text{if $\lambda\neq \kappa$} \quad \text{and} \quad |a_\kappa b_\kappa|=1.
$$
Then, given an arbitrary family $\{X_\lambda\, : \,
\lambda\in\Lambda\}$ of Banach spaces,
$$
n\left(\Bigl[\bigoplus_{\lambda\in \Lambda} X_\lambda\Bigr]_{E}\right) \,
\geq \, \inf\bigl\{n(X_\lambda)\,:\,\lambda\in \Lambda\bigr\}.
$$
\end{theorem}

\begin{proof}
Write $X=\Bigl[\bigoplus_{\lambda\in \Lambda} X_\lambda\Bigr]_{E}$
and $X'=\Bigl[\bigoplus_{\lambda\in \Lambda}
X^*_\lambda\Bigr]_{E'}\subseteq X^*$. Consider the sets
\begin{align*}
\mathcal{A} & = \left\{(a_\lambda x_\lambda)\ :\ a_\lambda\in \R^+_0,\,x_\lambda \in S_{X_\lambda}\,
    \forall \lambda\in\Lambda,\ (a_\lambda)\in A\right\}\subset S_X,\\
\mathcal{B} & = \left\{(b_\lambda x^*_\lambda)\ :\ b_\lambda\in \R^+_0,\,x^*_\lambda \in S_{X^*_\lambda}\,
    \forall \lambda\in\Lambda,\ (b_\lambda)\in B\right\}\subset S_{X'}.
\end{align*}
Then, it is clear that $\cconv(\mathcal{A})=B_X$ and that
$\mathcal{B}$ is norming for $X$.

Fix $T\in L(X)$ and $\varepsilon>0$, and write $T=(T_\lambda)$
where $T_\lambda=P_\lambda T \in L(X,X_\lambda)$. We may find
$x=(a_\lambda x_\lambda)\in \mathcal{A}$ and $x^*=(b_\lambda
x_\lambda^*)\in \mathcal{B}$ such that
$$
\|T\| -\varepsilon<|x^*(Tx)|=
\left|\sum_{\lambda\in \Lambda} b_\lambda x_\lambda^*
\bigl(T_\lambda\bigl((a_\lambda x_\lambda)_{\lambda\in \Lambda}\bigr)\bigr) \right|.
$$
By hypothesis, there is $\kappa\in \Lambda$ such that
\begin{equation}\label{eq:aibi=0}
a_\lambda b_\lambda=0 \text{ if $\lambda\neq \kappa$} \quad \text{and} \quad a_\kappa b_\kappa=1,
\end{equation}
and using the Bishop-Phelps theorem, we may and do suppose that
$x_\kappa^*\in S_{X_\kappa^*}$ attains its norm on an element
$\widetilde{x}_\kappa\in S_{X_\kappa}$. We also take
$y_\kappa^*\in S_{X_\kappa^*}$ such that $y_\kappa^*(x_\kappa)=1$.
For every $z\in X_\kappa$, we define $\Phi(z)\in X$ by
$$
[\Phi(z)]_\lambda=a_\lambda\, y_\kappa^*(z)\, x_\lambda \ \ \text{if $\lambda\neq \kappa$,}
\quad \text{and} \quad [\Phi(z)]_\kappa=a_\kappa\, z,
$$
which is well-defined since the norm of $E$ is absolute, satisfies
$\|\Phi(z)\|\leq \|z\|$ for every $z\in X_\kappa$ and that
$\Phi(z)$ is linear in $z$. Also, it is clear that
$\Phi(x_\kappa)=x$.

We consider the operator $S\in L(X_\kappa)$ given by
\begin{align*}
S(z) &=\left[\sum_{\lambda \neq \kappa} b_\lambda x_\lambda^* \bigl(T_\lambda(\Phi(z))\bigr)\right]
 \widetilde{x}_\kappa \ + \ b_\kappa T_\kappa(\Phi(z))
 \qquad \bigl(z\in X_\kappa\bigr)
\end{align*}
and observe that
$$
|x_\kappa^*(Sx_\kappa)|= \left|\sum_{\lambda\neq \kappa} b_\lambda x_\lambda^* \bigl(T_\lambda(\Phi(x_\kappa))\bigr)
 + b_\kappa x_\kappa^*\bigl(T_\kappa(\Phi(x_\kappa))\bigr) \right|=|x^*(Tx)|>\|T\|-\eps,
$$
so $\|S\|>\|T\|-\eps$. It follows that $v(S)>
n(X_\kappa)(\|T\|-\eps)$ and so there is $(\zeta,\zeta^*)\in
\Pi(X_\kappa)$ such that
\begin{equation}\label{eq:sums-reversed-inequality-S}
|\zeta^*(S \zeta)|\geq n(X_\kappa)(\|T\|-\eps).
\end{equation}
Now, we consider $\Psi(\zeta^*)\in X'\subset X^*$ given by
$$
[\Psi(\zeta^*)]_\lambda=
b_\lambda\,\zeta^*(\widetilde{x}_\kappa)\,x_\lambda^* \ \ \text{if $\lambda\neq \kappa$} \quad \text{and} \quad
[\Psi(\zeta^*)]_\kappa=b_\kappa \zeta^*,
$$
which is well-defined since $E'$ has absolute norm, and satisfies
$\|\Psi(\zeta^*)\|_{X^*} \leq 1$. We observe that,
by~\eqref{eq:aibi=0},
$$
[\Psi(\zeta^*)](\Phi(\zeta))=
\sum_{\lambda \neq \kappa} b_\lambda a_\lambda \zeta^*(\widetilde{x}_\kappa)y^*_\kappa(\zeta)x_\lambda^*(x_\lambda)
 + b_\kappa a_\kappa \zeta^*(\zeta) =
b_\kappa a_\kappa \zeta^*(\zeta)=\zeta^*(\zeta)=1
$$
and that
\begin{align*} [\Psi(\zeta^*)]\bigl(T(\Phi(\zeta))\bigr)  =
\left|\left[\sum_{\lambda\neq \kappa} b_\lambda\, x_\lambda^*
\bigl(T_\lambda(\Phi(\zeta))\bigr)\right]\zeta^*(\widetilde{x}_\kappa) + b_\kappa \zeta^*\bigl(
T_\kappa(\Phi(\zeta))\bigr) \right| = |\zeta^*(S \zeta)|.
\end{align*}
It then follows from \eqref{eq:sums-reversed-inequality-S} that
$$
v(T)\geq n(X_\kappa)(\|T\|-\eps) \geq \inf\{n(X_\lambda)\,:\, \lambda\in \Lambda\} (\|T\|-\eps).
$$
Letting $\eps\downarrow 0$ and considering all $T\in L(X)$, we get
$n(X)\geq \inf\{n(X_\lambda)\,:\, \lambda\in \Lambda\}$.
\end{proof}

Let us list the main consequences of the above theorem. The first
one gives a formula for the numerical index of $\ell_1$-sums and
$\ell_\infty$-sums. This result appeared in
\cite[Proposition~1]{M-P}.

\begin{corollary}\label{cor:equality-l1-linfty}
Let $\Lambda$ be a nonempty set and let $\{X_\lambda\, : \,
\lambda\in\Lambda\}$ be an arbitrary family of Banach spaces. Then
$$
n\left(\Bigl[\bigoplus_{\lambda\in \Lambda}
X_\lambda\Bigr]_{\ell_1}\right)   = n\left(\Bigl[\bigoplus_{\lambda\in \Lambda}
X_\lambda\Bigr]_{\ell_\infty}\right) =
\inf\bigl\{n(X_\lambda)\,:\,\lambda\in \Lambda\bigr\}.
$$
\end{corollary}

\begin{proof}
One inequality was proved in
Corollary~\ref{cor:c0-lp-sum-easyinequality}. To get the reversed
inequality, we just show that
Theorem~\ref{thm:absolutesum-equality} is applicable. For
$E=\ell_1(\Lambda)$, we consider
$$
A=\{\chi_{\{\lambda\}}\ : \ \lambda\in\Lambda\}\subset S_E \quad \text{and} \quad
B=\{(b_\lambda)\ : \ |b_\lambda|=1 \ \forall\lambda\in \Lambda\}\subset S_{E'}.
$$
Then it is immediate that $\cconv(A)=S_E$, that $B$ is norming for
$E$ and that given $a\in A$ and $b\in B$, there is $\kappa\in
\Lambda$ such that
\begin{equation*}%\label{eq:l1-linfty}
a_\lambda b_\lambda=0 \ \ \ \text{if $\lambda\neq \kappa$} \quad \text{and} \quad |a_\kappa b_\kappa|=1.
\end{equation*}
For $E=\ell_\infty(\Lambda)$, we interchange the roles of the sets
above and consider
$$
A=\{(a_\lambda)\ : \ |a_\lambda|=1 \ \forall\lambda\in \Lambda\}\subset S_{E} \quad \text{and} \quad
B=\{\chi_{\{\lambda\}}\ : \ \lambda\in\Lambda\}\subset S_{\ell_1(\Lambda)}\subset S_{E'}.
$$
Now, $B$ is trivially norming for $\ell_\infty(\Lambda)$ and
$\cconv(A)=B_E$ (indeed, the set of those norm-one functions which
have finitely many values is dense in $S_E$ by Lebesgue theorem
and the fact that a function taking finitely many values belongs
to $\conv(A)$ is easily proved by induction on the number of
values). Finally, it is immediate that given $a\in A$ and $b\in
B$, there is $\kappa\in \Lambda$ such that
\begin{equation*}%\label{eq:l1-linfty}
a_\lambda b_\lambda=0 \ \ \ \text{if $\lambda\neq \kappa$} \quad \text{and} \quad |a_\kappa b_\kappa|=1.\qedhere
\end{equation*}
\end{proof}

\begin{remark}
For $c_0$-sums the result above is also true but it is not
possible to prove it using
Theorem~\ref{thm:absolutesum-equality} (it is not possible to
find sets $A\subset S_E$ and $B\subset S_{E'}$ like there in
$E=c_0(\Lambda)$ unless it is finite-dimensional). We have to
wait until Section~\ref{sec:union} to provide a proof.
\end{remark}

Next result allows to calculate the numerical index of $E$-sums
of Banach spaces when $E$ has the RNP and $n(E)=1$. We will use
very recent results of H.-J.~Lee and the first and second named
authors of this paper \cite{LeeMartinMeri}.

\begin{corollary}\label{cor:RNP}
Let $\Lambda$ be a non-empty set and let $E$ be a subspace or
$\R^\Lambda$ endowed with an absolute norm. Suppose $E$ has the
RNP and $n(E)=1$. Then, given an arbitrary family $\{X_\lambda\, :
\, \lambda\in\Lambda\}$ of Banach spaces,
$$
n\left(\Bigl[\bigoplus_{\lambda\in \Lambda} X_\lambda\Bigr]_{E}\right) \,
= \, \inf\bigl\{n(X_\lambda)\,:\,\lambda\in \Lambda\bigr\}.
$$
\end{corollary}

\begin{proof}
Since $E$ has the RNP, it is order continuous (it does not
contain $\ell_\infty$) and so $E^*=E'$. Then inequality $\leq$
follows from Corollary~\ref{cor:leq-Eprima-F}. Let us prove the
reversed inequality. Let $A$ be the set of denting points of
$B_E$ and let $B$ be the set of extreme points of $B_{E'}$. It
follows from the RNP that $\cconv(A)=B_E$ and Krein-Milman
theorem gives that $B$ is norming. It is shown in
\cite{LeeMartinMeri} that, under these hypotheses, given
$a=(a_\lambda)\in A$ and $b=(b_\lambda)\in B$, there is
$\kappa\in \Lambda$ such that
$$
a_\lambda b_\lambda=0 \ \ \ \text{if $\lambda\neq \kappa$} \quad \text{and} \quad |a_\kappa b_\kappa|=1.
$$
This allows us to use Theorem~\ref{thm:absolutesum-equality} to
get the desired inequality.
\end{proof}

\begin{remark}
It follows from the proof of the above result that the
hypothesis of RNP is not needed in its full generality.
Actually, only two facts are needed: that $n(E)=1$ and that
$\cconv(A)=B_E$ where $A$ is the set of denting points of
$B_E$.
\end{remark}

Let us particularize here the above result for Banach spaces with
one-unconditional basis (finite or infinite).

\begin{corollary}$ $\label{cor:abs-sums-equality-basis}
\begin{itemize}
\item[(a)] Let $E$ be $\R^m$ endowed with an absolute norm
    such that $n(E)=1$, and let $X_1,\ldots,X_m$ be Banach
    spaces. Then
    $$
    n\left(\bigl[X_1\oplus \cdots \oplus X_m\bigr]_{E} \right)
        = \min\bigl\{n(X_1),\ldots,n(X_m)\bigr\}.
    $$
\item[(b)] Let $E$ be a Banach space with one-unconditional
    basis, having the RNP and such that $n(E)=1$. Then, given
    an arbitrary sequence $\{X_j\, : \, j\in \N\}$ of Banach
    spaces,
$$
n\left(\Bigl[\bigoplus_{j\in \N} X_j\Bigr]_{E}\right) \,
= \, \inf\bigl\{n(X_j)\,:\,j\in \N\bigr\}.
$$
\end{itemize}
\end{corollary}

Let us comment that the proof of Corollary~\ref{cor:RNP} in the
case of a finite-dimensional space can be done using results of
S.~Reisner \cite{Reis} and so, in this case, the very recent
reference \cite{LeeMartinMeri} is not needed.

\section{Numerical index and one-complemented subspaces}\label{sec:subspaces-examples}
One may wonder if there is any general inequality between the
numerical index of a Banach space and the numerical indices of its
subspaces (or of some kind of subspaces). Since $n(C(K))=1$ and
every Banach space contains (isometrically) a $C(K)$-space as
subspace (maybe with dimension one) and it is contained
(isometrically) in a $C(K)$ space (Banach-Mazur theorem), it is
not possible to get any general inequality. If we restrict
ourselves to special kind of subspaces, we may show a positive
result. Indeed, item (a) of
Corollary~\ref{cor:One-unconditonalbasis-easy} for $m=2$ shows
that the numerical index of a Banach space which is the absolute
sum of two subspaces is less or equal than the numerical index of
the subspaces. Let us comment that in this case the absolute sum
can be written in a different form. Indeed, suppose we have a
Banach space $X$ and two subspaces $Y$ and $Z$ such that
$X=Y\oplus Z$ and, for every $y\in Y$ and $z\in Z$, the norm of
$y+z$ only depends on $\|y\|$ and $\|z\|$. In such a case, it is
known that there exists an absolute norm $|\cdot|$ on $\R^2$ such
that
$$
\|x+z\|=|(\|x\|,\|z\|)| \qquad \bigl(x\in X,\ z\in Z\bigr),
$$
i.e.\ $X\equiv [Y\oplus Z]_{E}$ for $E=(\R^2,|\cdot|)$ and so
Corollary~\ref{cor:One-unconditonalbasis-easy} applies. We refer
the reader to \cite[\S~21]{B-D2} and \cite{MPRY} for background.

\begin{corollary}\label{cor:EMA}
Let $X$ be a Banach space and let $Y$, $Z$ be closed subspaces of
$X$ such that $X=Y\oplus Z$ and, for every $y\in Y$ and $z\in Z$,
$\|y+z\|$ only depends on $\|y\|$ and $\|z\|$. Then
$$
n(X)\leqslant \min\big\{n(Y),\ n(Z)\big\}.
$$
\end{corollary}

Let us comment that the above corollary already appeared in the
PhD dissertation (2000) of the first named author and was
published (in Spanish) in \cite[Proposici\'{o}n~1]{MarEMA}. Also,
Corollary~\ref{cor:c0-lp-sum-easyinequality} follows from the
above corollary since $c_0$-sums and $\ell_p$-sums are
associative (i.e.\ the whole sum is the $c_0$-sum or
$\ell_p$-sum of each summand and the sum of the rest of
summands). This was the way in which this result was proved in
\cite{M-P}. As a matter of facts, let us comment that the
unique associative absolute sums are $c_0$-sums and
$\ell_p$-sums \cite{Behetal}, and so
Theorem~\ref{thm:EMA-infinito} does not follow from the already
known Corollary~\ref{cor:EMA}.

It is natural to ask whether it would be possible that the
hypothesis of absoluteness in Corollary~\ref{cor:EMA} can be
weakened to general one-complemented subspaces, but we will show
that it is not possible. Moreover, we will show that the numerical
index of unconditional sums need not be smaller than the numerical
indices of the summands, even for projections associated to a
one-unconditional and one-symmetric norms in a three-dimensional
space. We recall that a closed subspace $Y$ of a Banach space $X$
is said to be an \emph{unconditional summand} of $X$ if there
exists another closed subspace $Z$ such that $X=Y\oplus Z$ and
$\|y+z\|=\|y + \theta z\|$ for every $y\in Y$, $z\in Z$ and
$|\theta|=1$. It is also said that $X$ is the \emph{unconditional
sum} of $Y$ and $Z$. When both $Y$ and $Z$ are one-dimensional, an
unconditional sum is actually an absolute sum, but this is not
true for higher dimensions.

\begin{example}\label{example:symmetric-1}
Let $X$ be the space $\R^3$ endowed with the norm
$$
\|(x,y,z)\|=\max\left\{\sqrt{x^2+y^2},\,\sqrt{x^2+z^2},\,
\sqrt{y^2+z^2}\right\} \qquad (x,y,z)\in\R^3.
$$
Then, the usual basis is one-unconditional and one-symmetric
for $X$, $n(X)>0$ but $n\bigl(P_2(X)\bigr)=0$ ($P_2$ is the
projection on the subspace of vectors supported on the first
two coordinates).
\end{example}

\begin{proof}
It is clear that $P_2(X)$ is isometrically isomorphic to the
two-dimensional Hilbert space, so $n\bigl(P_2(X)\bigr)=0$. Since
$X$ is finite-dimensional, to prove that $n(X)>0$ it is enough to
show that the unique operator $T\in L(X)$ with $v(T)=0$ is $T=0$.
Let $T$ be an operator with $v(T)=0$ represented by the matrix
$(a_{ij})$. Consider the following norm-one elements in $X$ and
$X^*$:
\begin{align*}
x_1&=(1,0,0), \quad x_2=(0,1,0), \quad x_3=(0,0,1), \quad
x_4=\frac{1}{\sqrt{2}}(1,1,0),\\
x_5&=\frac{1}{\sqrt{2}}(0,1,1), \quad
x_6=\frac{1}{\sqrt{2}}(1,0,1),\quad x_7=\frac{1}{\sqrt{2}}(1,-1,1),
\quad
x_8=\frac{1}{\sqrt{2}}(1,1,1)\\
x_1^*&=(1,0,0), \quad x_2^*=(0,1,0),\quad x_3^*=(0,0,1), \\
x_4^*&=\frac{1}{\sqrt{2}}(1,1,0),\quad
x_5^*=\frac{1}{\sqrt{2}}(0,1,1),\quad
x_6^*=\frac{1}{\sqrt{2}}(1,0,1),\quad
x_7^*=\frac{1}{3\sqrt{2}}(1,-1,0)+\frac{2}{3\sqrt{2}}(0,-1,1).
\end{align*}
Observe that $\|x_7^*\|\leq 1$ and $x_7^*(x_7)=1$.

Since $x_i^*(x_i)=1$ for $i=1,2,3$ we have that
$a_{ii}=x_i^*(Tx_i)=0$ for $i=1,2,3$. Analogously, using that
$$
x_4^*(x_4)=1=x_5^*(x_5)=x_6^*(x_6)=x^*_4(x_8)=x_6^*(x_8)=x_7^*(x_7)
$$
we obtain the following restraints on $a_{ij}$\,:
\begin{align*}
0&=x_4^*(Tx_4)=\frac12(a_{12}+a_{21}) \text{ which implies }
a_{21}=-a_{12},\\
0&=x_5^*(Tx_5)=\frac12(a_{23}+a_{32}) \text{ which implies }
a_{32}=-a_{23},\\
0&=x_6^*(Tx_6)=\frac12(a_{13}+a_{31}) \text{ which implies }
a_{31}=-a_{13},\\
0&=x_5^*(Tx_8)=\frac12 (-a_{12}-a_{13})\text{ which implies }
a_{13}=-a_{12},\\
0&=x_6^*(Tx_8)=\frac12(a_{12}-a_{23})\text{ which implies }
a_{23}=a_{12},\\
0&=x_7^*(Tx_7)=\frac13a_{12}.
\end{align*}
Therefore, we have that $T=0$.
\end{proof}

We do not know whether the example given above can be adapted to
the complex case. Nevertheless, we are able to show a complex
example with one-unconditional (not symmetric) norm.

\begin{example}\label{example:complex-unconditional}
Let us consider the normed space $X$ to be $\K^5$ ($\K=\R$ or
$\K=\C$) endowed with the norm
$$
\|(x_1,x_2,x_3,x_4,x_5)\|=\max\bigl\{|x_1|+|x_2|, |x_2|+|x_3|+|x_5|,|x_3|+|x_4|\bigr\}
\qquad (x_1,x_2,x_3,x_4,x_5)\in \K^5.
$$
Then, the usual basis is one-unconditional for $X$, $n(X)=1$ and
$n\bigl(P_4(X)\bigr)<1$.
\end{example}

\begin{proof}
For the real case, it was proved in \cite[\S 3]{Reis} that $X$ is
a so-called CL-space and that $P_4(X)$ is not. In the
finite-dimensional case, this is equivalent to say that $n(X)=1$
and $n\bigl(P_4(X)\bigr)<1$ (see \cite[\S 3]{KaMaPa}, for
instance). For the complex case, it was shown in
\cite[Proposition~4.3]{Kim-LeeJFA} that the (natural)
complexification of a $n$-dimensional normed space with absolute
norm (i.e.\ the usual basis of $\R^n$ is one-unconditional) is a
CL-space if and only if the real version is a CL-space. Therefore,
$X$ is a (complex) CL-space and $P_4(X)$ is not. As in the real
case, this gives that $n(X)=1$ and $n\bigl(P_4(X)\bigr)<1$.
\end{proof}

Using the continuity of the numerical index with respect to the
Banach-Mazur distance for equivalent norms \cite{F-M-P}, it is
possible to obtain examples as in \ref{example:symmetric-1} and
\ref{example:complex-unconditional} which are uniformly convex and
uniformly smooth.

\pagebreak[2]

\begin{examples}$ $\label{example:symmetric-smooth}
\begin{itemize}
\item[(a)] \textsc{Real case:\ } For $p\geq1$, we consider
    $X_p=(\R^3, \|\cdot\|_{(p)})$ where
$$
\|(x,y,z)\|_{(p)}=2^{-\frac{1}{p}}\,\left((x^2+y^2)^\frac{p}{2}+(x^2+z^2)^\frac{p}{2} +
(y^2+z^2)^\frac{p}{2}\right)^\frac1p
\qquad (x,y,z)\in\R^3.
$$
Then, the usual basis is one-unconditional and one-symmetric
for every $X_p$ and $P_2(X_p)\equiv (\R^2, \|\cdot\|_{(p)})$
where
$$
\|(x,y)\|_{(p)}=2^{-\frac{1}{p}}\left((x^2+y^2)^\frac{p}{2}+|x|^p+|y|^p\right)^\frac1p
\qquad (x,y)\in\R^2.
$$
By the continuity of the numerical index with respect to the
Banach-Mazur distance for equivalent norms
\cite[Proposition~2]{F-M-P}, one has that
$$
\lim_{p\rightarrow \infty} n(Y_p)=n(X)>0 \qquad \text{and} \qquad
\lim_{p\rightarrow \infty}
n\bigl(P_2(Y_p)\bigr)=n(H_2)=0
$$
where $X$ is the three-dimensional space constructed in
Example~\ref{example:symmetric-1} and $H_2$ is the
two-dimensional real Hilbert space.

Therefore, for $p$ big enough, the uniformly convex and
uniformly smooth space $X_p$ satisfies
$$
n(X_p)>\frac12 n(X)\qquad \text{and} \qquad
n\bigl(P_2(Y_p)\bigr)<\frac12 n(X).
$$
Let us comment that the spaces $X_p$ are actually Lorentz
spaces.
\item[(b)] \textsc{Complex case:\ } Using the same kind of
    tricks that the ones given above, it is possible to
    produce uniformly convex and uniformly smooth versions of
    Example~\ref{example:complex-unconditional}. Therefore,
    there is a $5$-dimensional complex uniformly convex and
    uniformly smooth normed space $X$ with one-unconditional
    basis such that $n(X)>n\bigl(P_4(X)\bigr)$.
\end{itemize}
\end{examples}

\section{K\"{o}the-Bochner function spaces}\label{sec:kothe}
Let $(\Omega,\Sigma,\mu)$ be a complete $\sigma$-finite measure
space. We denote by $L_0(\mu)$ the space of all (equivalent
classes modulo equality a.e.\ of) $\Sigma$-measurable locally
integrable real-valued functions on $\Omega$. A \emph{K\"{o}the
function space} is a linear subspace $E$ of $L_0(\mu)$ endowed
with a complete norm $\|\cdot\|_E$ and satisfying the following
conditions.
\begin{enumerate}
\item[(i)] If $|f|\leq |g|$ a.e.\ on $\Omega$, $g\in E$ and
    $f\in L_0$, then $f\in E$ and $\|f\|_E\leq \|g\|_E$.
\item[(ii)] For every $A\in \Sigma$ with $0<\mu(A)<\infty$,
    the characteristic function $\chi_A$ belongs to $E$.
\end{enumerate}
We refer the reader to the classical book by J.~Lindenstrauss and
L.~Tzafriri \cite{LTII} for more information and background on
K\"{o}the function spaces. Let us recall some useful facts about these
spaces which we will use in the sequel. First, $E$ is a Banach
lattice in the pointwise order. The \emph{K\"{o}the dual} $E'$ of $E$
is the function space defined as
$$
E'=\left\{g\in L_0(\mu)\ : \ \|g\|_{E'}:=
\sup_{f\in B_E}\int_\Omega |fg|\,d\mu<\infty \right\},
$$
which is again a K\"{o}the space on $(\Omega,\Sigma,\mu)$. Every
element $g\in E'$ defines naturally a continuous linear functional
on $E$ by the formula
$$
f\longmapsto \int_\Omega fg\, d\mu \qquad (f\in E),
$$
so we have $E'\subseteq E^*$ and this inclusion is isometric.

Let $E$ be a K\"{o}the space on a complete $\sigma$-finite measure
space $(\Omega,\Sigma,\mu)$ and let $X$ be a real or complex
Banach space. A function $f:\Omega\longrightarrow X$ is said to be
\emph{simple} if $f=\sum_{i=1}^n x_i\,\chi_{A_i}$ for some
$x_1,\ldots,x_n\in X$ and some $A_1,\ldots,A_n\in \Sigma$. The
function $f$ is said to be \emph{strongly measurable} if there is
a sequence of simple functions $\{f_n\}$ with $\lim
\|f_n(t)-f(t)\|_{X}=0$ for almost all $t\in \Omega$. We write
$E(X)$ for the space of those strongly measurable functions
$f:\Omega \longrightarrow X$ such that the function
$$
t\longmapsto \|f(t)\|_X \qquad (t\in \Omega)
$$
belongs to $E$, and we endow $E(X)$ with the norm
$$
\|f\|_{E(X)}=\bigl\|t\longmapsto \|f(t)\|_X\bigr\|_E.
$$
Then $E(X)$ is a real or complex (depending on $X$) Banach space
and it is called a \emph{K\"{o}the-Bochner function space}. We refer
the reader to the recent book by P.-K.~Lin \cite{LinKothe} for
background. Let us introduce some notation and recall some useful
facts which we will use in the sequel. For an element $f\in E(X)$
we write $|f|\in E$ for the function $|f|(\cdot)=\|f(\cdot)\|_X$,
we consider a measurable function $\widetilde{f}:\Omega
\longrightarrow S_X$ such that $f=|f|\,\widetilde{f}$ a.e.\ and we
observe that $\|f\|_{E(X)}=\|\,|f|\,\|_E$.

We write $E'(X^*,w^*)$ to denote the space of $w^*$-scalarly
measurable functions $\Phi:\Omega\longrightarrow X^*$ such that
$\|\Phi(\cdot)\|_{X^*}\in E'$, which act on $E(X)$ as integral
functionals:
$$
\langle \Phi,f\rangle = \int_\Omega \langle \Phi(t),f(t)\rangle\,d\mu(t) \qquad \bigl(f\in E(X)\bigr).
$$
For an integral functional $\Phi\in E'(X^*,w^*)$, we write
$|\Phi|\in E'$ for the function
$|\Phi|(\cdot)=\|\Phi(\cdot)\|_{X^*}$ and we consider a
$w^*$-scalarly measurable function $\widetilde{\Phi}:\Omega
\longrightarrow S_{X^*}$ such that
$\Phi=|\Phi|\,\widetilde{\Phi}$.

Our first result gives an inequality between the numerical index
of $E(X)$ and $n(X)$, provided that there are sufficiently many
integral functionals.

\begin{theorem}\label{theorem:Kothe-vector}
Let $(\Omega,\Sigma,\mu)$ be a complete $\sigma$-finite measure
space, let $E$ be a K\"{o}the space on $(\Omega,\Sigma,\mu)$, and let
$X$ be a Banach space. Suppose that there is a dense subset
$\mathcal{A}$ of $S_{E(X)}$ such that for every $f\in \mathcal{A}$
there is $\Phi_f\in E'(X^*,w^*)$ satisfying
$\|\,|\Phi_f|\,\|_{E'}=\langle \Phi_f\, , \, f\rangle=1$. Then
$$
n\bigl(E(X)\bigr)\leq n(X).
$$
\end{theorem}

\begin{proof}
We take an operator $S\in L(X)$ with $\|S\|=1$, and define $T\in
L(E(X))$ by
$$
[T(f)](t)=S(f(t))\ \Bigl[=|f|(t)\,S(\widetilde{f}(t))\Bigr]
\qquad \bigl(t\in\Omega,\ f\in E(X)\bigr).
$$
We claim that $T$ is well-defined and $\|T\|=1$. Indeed, for $f\in
E(X)$, $T(f)$ is strongly measurable and
$$
\|[T(f)](t)\|_{X} = |f|(t)\,\|S(\widetilde{f}(t))\| \leq |f|(t) \qquad (t\in \Omega),
$$
so $T(f)\in E(X)$ with $\|T(f)\|_{E(X)} \leq \|\,|f|\,\|_E
=\|f\|_{E(X)}$. This gives $\|T\|\leq 1$. Conversely, we fix $A\in
\Sigma$ such that $0<\mu(A)<\infty$ and for each $x\in S_X$
consider $f=\|\chi_A\|_E^{-1}\,x\,\chi_A\in S_{E(X)}$. Then,
$\|f\|=1$ and
$$
\|[T(f)](t)\|_{X} =\frac{\chi_A(t)\|S(x)\|_X}{\|\chi_A\|_E}, \quad
\text{ so } \quad \|T\|\geq \|T(f)\|_{E(X)} =
\left\|\frac{\chi_A\,\|S(x)\|_X}{\|\chi_A\|_E}\right\|_E \geq \|S(x)\|_X.
$$
By just taking supremum on $x\in S_X$, we get $\|T\|\geq \|S\|=1$
as desired.

Next, we consider $f\in \mathcal{A}$ and observe that
\begin{align*}
1=\langle \Phi_f , f\rangle & = \int_\Omega \langle \Phi_f(t),f(t)\rangle\, d\mu(t) =
 \int_\Omega |\Phi_f|(t)\,|f|(t)\,\langle \widetilde{\Phi_f}(t),\widetilde{f}(t)\rangle\ d\mu(t) \\
 & \leq \int_\Omega |\Phi_f|(t)\,|f|(t)\ d\mu(t) \leq \langle |\Phi_f|,|f|\rangle
 \leq \|\,|\Phi_f|\,\|_{E'}\,\|\,|f|\,\|_{E}=1.
\end{align*}
It follows that
\begin{equation*}%\label{eq:Kothe-functional}
\langle \widetilde{\Phi_f}(t),\widetilde{f}(t)\rangle = 1 \ \
\text{a.e.} \qquad \text{and} \qquad \int_\Omega
|\Phi_f|(t)\,|f|(t)\ d\mu(t)=1.
\end{equation*}
On the other hand,
\begin{align*}
\left|\langle \Phi_f, T(f)\rangle\right| & =
\left|\int_\Omega \langle \Phi_f(t),S(f(t))\rangle\ d\mu(t) \right|  = \left|\int_\Omega |\Phi_f|(t)\,|f|(t)\,
\langle \widetilde{\Phi_f}(t),S(\widetilde{f}(t))\rangle\ d\mu(t) \right| \\
 & \leq \int_\Omega |\Phi_f|(t)\,|f|(t)\,
\bigl|\langle \widetilde{\Phi_f}(t),S(\widetilde{f}(t))\rangle\bigr|\ d\mu(t) \leq
\int_\Omega |\Phi_f|(t)\,|f|(t)\,v(S)\, d\mu(t)=v(S).
\end{align*}
Since $\mathcal{A}$ is dense in the unit sphere of $E(X)$, it
follows from Lemma~\ref{lemma:Theorem93} that the above inequality
implies that $v(T)\leq v(S)$. Therefore, $n\bigl(E(X)\bigr) \leq
v(S)$. In view of the arbitrariness of $S\in L(X)$ with $\|S\|=1$,
we get $n\bigl(E(X)\bigr) \leq n(X)$, as desired.
\end{proof}

The main application of the above theorem concerns order
continuous K\"{o}the spaces. We say that a K\"{o}the space $E$ is
\emph{order continuous} if $0\leq x_\alpha \downarrow 0$ and
$x_\alpha\in E$ imply that $\lim \|x_\alpha\|=0$ (this is known
to be equivalent to the fact that $E$ does not contain an
isomorphic copy of $\ell_\infty$). If $E$ is order continuous,
the inclusion $E'\subseteq E^*$ is surjective (so $E^*$
completely identifies with $E'$) and the set of those simple
functions belonging to $E(X)$ is norm-dense in $E(X)$.

\begin{corollary}\label{cor:Kothe-vector}
Let $(\Omega,\Sigma,\mu)$ be a complete $\sigma$-finite measure
space and let $E$ be an order continuous K\"{o}the space. Then, for
every Banach space $X$,
$$
n\bigl(E(X)\bigr)\leq n(X).
$$
\end{corollary}

\begin{proof}
We consider $\mathcal{A}$ to denote the set of norm-one simple
functions belonging to $E(X)$. Consider $f\in\mathcal{A}$, which
is of the form
\begin{equation*}%\label{eq:kothe-simple}
f = \sum\limits_{j=1}^m a_j \,x_j\, \chi_{{A_j}} \in
E(X),
\end{equation*}
where $m \in \mathbb N$,  $a_j \geq 0$, $x_j\in S_{X}$,
$A_1,\dots, A_m\in\Sigma$ are pairwise disjoint, and $|f|=\sum a_j
\chi_{A_j}\in S_E$. Since $E$ is order-continuous, $E^*= E'$ and
so we may find a positive function $\varphi \in E'$ with
$\|\varphi\|_{E'}=1$ such that
$$
1=\langle f,\varphi \rangle = \int_\Omega f\varphi\,d\mu =
 \sum_{j=1}^m \int_{A_j} \varphi(t) a_j \, d\mu(t).
$$
Now, for $j=1,\ldots,m$, we choose $x_j^*\in S_{X^*}$ such that
$x_j^*(x_j)=1$ and consider $\Phi_f\in E'(X^*,w^*)$ defined by
\begin{equation*}%\label{eq:kothe-def-functional}
\langle \Phi_f, g \rangle =\sum_{j=1}^m \int_{A_j} \varphi(t) x_j^*\bigl(g(t)\bigr)\,d\mu(t)
\qquad \bigl(g\in E(X)\bigr).
\end{equation*}
Then, we have $\|\,|\Phi|\,\|_{E'}=\|\varphi\|_{E'}=1$ and also
\begin{equation*}%\label{eq:kothe-Phiisnorming}
\langle \Phi_f,f\rangle = \sum_{j=1}^m \int_{A_j} \varphi(t) a_j x_j^*(x_j)\, d\mu(t) =
\sum_{j=1}^m \int_{A_j} \varphi(t) a_j \, d\mu(t)=1.
\end{equation*}
Since $E$ is order continuous, $\mathcal{A}$ is dense in
$S_{E(X)}$ and we may apply Theorem~\ref{theorem:Kothe-vector}.

Alternatively, we can prove this corollary by using the deep
result of the theory of K\"{o}the-Bochner spaces that for an order
continuous K\"{o}the space $E$ and a Banach space $X$, the whole
$E(X)^*$ identifies isometrically with $E'(X^*,w^*)$ (see
\cite[Theorem~3.2.4]{LinKothe}) and, therefore, for every $f\in
S_{E(X)}$ there is a norm-one element $\Phi_f$ of $E'(X^*,w^*)$
such that $\langle \Phi_f, f\rangle=1$.
\end{proof}

Since $L_p(\mu)$ spaces are order continuous K\"{o}the spaces for
$1\leq p <\infty$, as an immediate consequence of
Corollary~\ref{cor:Kothe-vector} we obtain the following
corollary. For $p=1$ it appeared in \cite{M-P} and for
$1<p<\infty$ it appeared in \cite{Eddari-Khamsi-Aksoy}.

\begin{corollary}\label{cor:Lp-vector}
Let $(\Omega,\Sigma,\mu)$ be a complete $\sigma$-finite measure
space, $1 \leq p<\infty$, and let $X$ be a Banach space. Then
$$
n\bigl(L_p(\mu,X)\bigr)\leq n(X).
$$
\end{corollary}

The K\"{o}the space $L_\infty(\mu)$ is not order continuous in the
infinite-dimensional case, so Corollary~\ref{cor:Kothe-vector}
does not cover this case. Anyway, we may apply directly
Theorem~\ref{theorem:Kothe-vector} to get the corresponding
result. The following statement was proved in
\cite[Theorem~2.3]{M-V}.

\begin{corollary}
Let $(\Omega,\Sigma,\mu)$ be a complete $\sigma$-finite measure
space and let $X$ be a Banach space. Then
$$
n\bigl(L_\infty(\mu,X)\bigr)\leq n(X).
$$
\end{corollary}

\begin{proof}
Since every element of $L_\infty(\mu,X)$ has essentially separable
range, it is possible to show that the subset
$$
\mathcal{A}=\{x\, \chi_A + g\, \chi_{\Omega\setminus A}\ : \ x\in S_X,\ g\in
B_{L_\infty(\mu,X)},\ A\in \Sigma\ \text{with}\ 0<\mu(A)<\infty\}
$$
is dense in the unit sphere of $L_\infty(\mu,X)$. For every $f\in
\mathcal{A}$, write $f=x\,\chi_A + g\chi_{\Omega\setminus A}$,
pick $x^*\in S_{X^*}$ such that $x^*(x)=1$ and observe that the
function
$$
\Phi_f=\frac{1}{\mu(A)}\,x^*\,\chi_A
$$
belongs to $L_1(\mu,X^*)\subset [L_\infty(\mu)]'(X^*,w^*)$, has
norm one and $\langle \Phi_f,f\rangle=1$. Then, the hypotheses of
Theorem~\ref{theorem:Kothe-vector} are satisfied and so
$n\bigl(L_\infty(\mu,X)\bigr)\leq n(X)$.
\end{proof}

For $E=L_1(\mu)$ and $E=L_\infty(\mu)$, it is actually known that
$n\bigl(E(X)\bigr)=n(X)$ for every Banach space $X$
\cite{M-P,M-V}. It would be interesting to study for which K\"{o}the
spaces $E$ the above equality also holds.

\section{Banach spaces with a dense increasing family of one-complemented
subspaces}\label{sec:union}

Our goal in this section is to show that the numerical index of
a Banach space which contains a dense increasing union of
one-complemented subspaces is greater or equal than the limit
superior of the numerical indices of those subspaces, and to
provide some consequences of this fact. We need some notation.
Recall that a \emph{directed set} (or \emph{filtered set}) is a
set $I$ endowed with an partial order $\leq$ such that for
every $i,j\in I$, there is $k\in I$ such that $i\leq k$ and
$j\leq k$.

\begin{theorem}\label{theorem:directedset}
Let $Z$ be a Banach space, let $I$ be a directed set, and let
$\{Z_i\,:\,i\in I\}$ be an increasing family of one-complemented
closed subspaces such that $Z=\overline{\bigcup\limits_{i \in I}
Z_i}$. Then,
$$
n(Z)\geq \limsup_{i\in I} n(Z_i).
$$
\end{theorem}

\begin{proof}
Let us denote by $P_i:Z\longrightarrow Z_i$ the norm-one
projection from $Z$ onto $Z_i$ and by $J_i:Z_i\longrightarrow Z$
the natural inclusion. Then,
$$
\|P_i\|=\|J_i\|=1, \quad P_i\circ J_i=\Id_{Z_i} \qquad (i\in I).
$$
We fix $\varepsilon>0$ and take an operator $T\in L(Z)$ such that
$$
\|T\|=1 \qquad \text{and} \qquad v(T)\leq (1+\varepsilon)\,n(Z).
$$
For each $i\in I$, consider the operator $S_i=P_i\, T\, J_i \in
L(Z_i)$.

\noindent\textit{Claim.}\ $v(S_i)\leq v(T)$.\newline Indeed, for
$(y,y^*)\in \Pi(Z_i)$, one has
\begin{align*}
|\langle y^*, S_i(y)\rangle|& = |\langle y^*, P_i(T(J_i(y))\rangle|  \\
& = |\langle P_i^*(y^*), T(J_i(y))\rangle | \leq v(T),
\end{align*}
where the last inequality follows from the fact that
$\|P_i^*(y^*)\|\leq 1$, $\|J_i(y)\|\leq 1$ and
$$
\langle P_i^*(y^*), J_i(y)\rangle =\langle y^*,
[P_i\circ J_i](y)\rangle = \langle y^*,y\rangle =1.
$$
\noindent\textit{Claim.}\ $\lim_{i\in I} \|S_i\|=1$.\newline
Indeed, on the one hand,
$$
\|S_i\|=\|P_i T J_i\|\leq \|P_i\|\|T\|\|J_i\|=1 \quad (i\in I).
$$
On the other hand, for $\delta>0$ fixed, we take $x\in S_Z$ such
that $\|Tx\|>1-\delta/4$. Since the family $\{Z_i\,:\,i\in I\}$ is
increasing and its union is dense, we may find $i_0\in I$ such
that for $i\geq i_0$ there are $y,z\in Z_i$ such that
\begin{equation*}%\label{eq:monotonic-1}
\|Tx-J_i(z)\|<\delta/4 \qquad \text{and} \qquad \|x-J_i(y)\|<\delta/4.
\end{equation*}
Then, we have that
$$
\|z\|=\|J_i(z)\|\geq \|Tx\|-\|J_i(z)-Tx\|>1-\delta/4 - \delta/4=1-\delta/2
$$
and
\begin{align*}
\|S_i(y)\| &= \|[P_i T J_i](y)\| \\
& = \bigl\|P_i(J_i(z)) - \bigl[P_i(J_i(z)) - P_i(Tx)\bigr] -
\bigl[P_i(Tx)-P_i(T(J_i(y)))\bigr]\bigr\| \\
& \geq \|z\|-\|P_i\|\|J_i(z)-Tx\| - \|P_i\|\|T\|\|x-J_i(y)\| \\
& > 1- \delta/2 -\delta/4 - \delta/4=1-\delta.
\end{align*}
Since $\|y\|=\|J_i(y)\|\leq 1+\delta/4$, it follows that
$$
\|S_i\|\geq \dfrac{1-\delta}{1+\delta/4}.
$$
This gives $\displaystyle \lim_{i\in I} \|S_i\|=1$, as claimed.

To finish the proof, we just observe that for every $i\in I$, we
have that
$$
(1+\varepsilon)n(Z)\geq v(T)\geq v(S_i) \geq n(Z_i)\,\|S_i\|
$$
and, therefore,
$$
(1+\varepsilon) n(Z)\geq \limsup_{i\in I} \bigl[n(Z_i)\,\|S_i\|\bigr]
= \limsup_{i\in I} n(Z_i)\,\lim_{i\in I} \|S_i\|= \limsup_{i\in I} n(Z_i).
$$
The result follows by just taking $\eps\downarrow 0$.
\end{proof}

The easiest particular case of the above result is to Banach
spaces with a monotone basis (i.e.\ a basis whose basic
constant is $1$).

\begin{corollary}\label{cor:monotonic-basis}
Let $Z$ be a Banach space with a monotone basis $(e_m)$ and for
each $m\in \N$, let $X_m=\linspan\{e_k\, : \, 1\leq k \leq
m\}$. Then
$$
n(X) \geq \limsup_{m\to \infty} n(X_m).
$$
\end{corollary}

It is known, see Section~\ref{sec:Lp}, that if $X=\ell_p$ then
the inequality above is actually an equality, but we do not
know whether the same is true in any other type of spaces.

\begin{problem}
Let $Z$ be a Banach space with a monotone (or even
one-unconditional, one-symmetric) basis $\{e_m\}_{m\in\N}$, and
let $X_m=\linspan\{e_k\, : \, 1\leq k \leq m\}$. Is it true
that $n(X) = \limsup\limits_{m\to \infty} n(X_m)$?
\end{problem}

Next examples show that the inequality in
Theorem~\ref{theorem:directedset} may be strict.

\begin{examples}$ $
\begin{itemize}
\item[(a)] \textsc{Real case:\ } Let $X$ be the
    three-dimensional real space given in
    Example~\ref{example:symmetric-1} such that
    $$
    n(X)>n\bigl(P_2(X)\bigr)=0,
    $$
    where $P_2(X)$ is the subspace of $X$ spanned by the
    two first coordinates. Now, consider the space
    $Z=\ell_1(X)$ and for each $m\in \N$, we consider the
    subspace
    \begin{align*}
    Z_m  & =\bigl\{x=(x_k)\in Z\ : \ x_{m+1}(3)=0,\ x_k=0
     \ \forall k\geq m+2 \bigr\} \\
    & \equiv X\oplus_1 \overset{m}{\cdots} \oplus_1 X \oplus_1 P_2(X).
    \end{align*}
    Then, we have $n(Z)=n(X)>0$ and
    $n(Z_m)=n\bigl(P_2(X)\bigr)=0$ by
    Corollary~\ref{cor:equality-l1-linfty}. Observe that
    each $Z_m$ is one-complemented in $Z$, the sequence
    $\{Z_m\}$ is increasing and
    $Z=\overline{\bigcup_{m\in\N} Z_m}$. But $n(Z)>
    \limsup_{m\to\infty} n(Z_m)$. Let us comment that the
    space $Z$ has a monotone basis and that the subspaces
    $Z_m$ are actually the range of some of the projections
    associated to the basis.
\item[(b)] \textsc{Complex case:\ } If we use the
    five-dimensional complex space $X$ of
    Example~\ref{example:complex-unconditional}, we can repeat
    the proof above to get the same kind of example in the
    complex case.
\end{itemize}
\end{examples}

Let us present here applications of
Theorem~\ref{theorem:directedset}. The first one allows to
calculate the numerical index of $L_1(\mu,X)$. This result
appeared in \cite[Theorem~8]{M-P}.

\begin{corollary}\label{cor:L_1muX-equality}
Let $(\Omega,\Sigma,\mu)$ be a complete positive measure space and
let $X$ be a Banach space. Then $n\bigl(L_1(\mu,X)\bigr)=n(X)$.
\end{corollary}

\begin{proof} Set $Z=L_1(\mu,X)$. We write $I$ for the
family of all finite collections of pairwise disjoint elements
of $\Sigma$ with finite measure, ordered by $\pi_1 \leq \pi_2$
if and only if each element in $\pi_1$ is a union of elements
in $\pi_2$. Then $I$ is a directed set. For $\pi\in I$, we
write $Z_\pi$ for the subspace of $Z$ consisting of all simple
functions supported in the elements of $\pi$. Now, for every
$\pi\in I$, the subspace $Z_\pi$ is isometrically isomorphic to
$\ell_1^m(X)$ ($m$ is the number of elements in $\pi$, see
\cite[Lemma~II.2.1]{D-U} for instance), it is one-complemented
by the conditional expectation associated to the partition
$\pi$ and, finally, the density of simple functions on
$L_1(\mu,X)$ gives that $Z= \overline{\bigcup_{\pi\in I}
Z_\pi}$. Then, Theorem~\ref{theorem:directedset} applies and it
follows that
\begin{equation*}
n\bigl(L_1(\mu,X)\bigr)\geq \limsup_{\pi\in I} n(Z_\pi)=\limsup_{\pi\in I} n\bigl(\ell_1^m(X)\bigr)
=n(X),
\end{equation*}
where the last equality above follows from
Corollary~\ref{cor:equality-l1-linfty}. To get the reversed
inequality, we start by using that there is a family of finite
measure spaces $\{\nu_j\,:\, j\in J\}$ such that
$$
L_1(\mu,X)\equiv \Bigl[\bigoplus_{j\in J} L_1(\nu_j,X)\Bigr]_{\ell_1}
$$
(for $X=\K$ there is a proof of this fact in \cite[p.~501]{D-F}
which obviously extends to any arbitrary space $X$). Then pick any
$j\in J$ and use Corollary~\ref{cor:c0-lp-sum-easyinequality} to
get that $n\bigl(L_1(\mu,X)\bigr)\leq n\bigl(L_1(\nu_j,X)\bigr)$.
Since $\nu_j$ is finite, we may now apply
Corollary~\ref{cor:Lp-vector} to get $n\bigl(L_1(\nu_j,X)\bigr)
\leq n(X)$.
\end{proof}

In Section~\ref{sec:Lp} we will use arguments similar to the
ones in the above proof (and some more) to get a result on
$n\bigl(L_p(\mu,X)\bigr)$ for $1<p<\infty$.

Next application of Theorem~\ref{theorem:directedset} is to
absolute sums of Banach spaces. As a particular case, we will
calculate the numerical index of a $c_0$-sum of Banach spaces. We
need some notation. Given a nonempty set $\Lambda$, let
$\mathcal{I}_\Lambda$ denote the collection of all finite subsets
of $\Lambda$ ordered by inclusion. Given a linear subspace $E$ of
$\R^\Lambda$ for every $A\in \mathcal{I}_\Lambda$ we consider
$E_A$ to be the subspace of $E$ consisting of those elements of
$E$ supported in $A$ and observe that $E_A$ is a linear subspace
of $\R^A$ with an absolute norm (the restriction of the one in
$E$).

\begin{corollary}\label{cor:equality-E=cupE_A}
Let $\Lambda$ be a nonempty set and let $E$ be a linear subspace
of $\R^\Lambda$ with an absolute norm. Suppose that there is a
subset $I$ of $\mathcal{I}_\Lambda$, which is still directed by
the inclusion, satisfying that $\bigcup_{A\in I} E_A$ is dense in
$E$ and that $n(E_A)=1$ for every $A\in I$. Then, given an
arbitrary family $\{X_\lambda\, : \, \lambda\in\Lambda\}$ of
Banach spaces,
$$
n\left(\Bigl[\bigoplus_{\lambda\in \Lambda} X_\lambda\Bigr]_{E}\right) \,
= \, \inf\bigl\{n(X_\lambda)\,:\,\lambda\in \Lambda\bigr\}.
$$
\end{corollary}

\begin{proof}
We write $Z=\Bigl[\bigoplus_{\lambda\in \Lambda}
X_\lambda\Bigr]_{E}$ and for each
$A=\{\lambda_1,\ldots,\lambda_m\}\in I$ we write $Z_A$ for the
$E$-sum of the family $\{Y_\lambda\}$ where $Y_\lambda=X_\lambda$
for $\lambda\in\{\lambda_1,\ldots,\lambda_m\}$ and
$Y_\lambda=\{0\}$ otherwise (we will use the agreement that
$n(\{0\})=1$). We observe that each $Z_A$ is one-complemented in
$Z$ and that $Z=\overline{\bigcup_{A\in I} Z_A}$. Therefore,
Theorem~\ref{theorem:directedset} applies and provides that
$$
n(Z)\geq \limsup_{A\in I} n(Z_A).
$$
Now, we take $A=\{\lambda_1,\ldots,\lambda_m\}\in I$, observe that
$E_A$ is isometrically isomorphic to $\R^m$ endowed with an
absolute norm which we denote by $\widetilde{E_A}$, and also
observe that $Z_A$ is isometrically isomorphic to the
$\widetilde{E_A}$-sum of $X_{\lambda_1},\ldots,X_{\lambda_m}$. As
$n(\widetilde{E_A})=n(E_A)=1$ by hypothesis, we may use
Corollary~\ref{cor:abs-sums-equality-basis}.a to get that
$$
n(Z_A)=n\left(\bigl[X_{\lambda_1}\oplus \cdots \oplus X_{\lambda_m}\bigr]_{\widetilde{E_A}} \right)
= \min\bigl\{n(X_{\lambda_1}),\ldots,n(X_{\lambda_m})\bigr\} \geq
\inf\bigl\{n(X_\lambda)\,:\,\lambda\in \Lambda\bigr\}.
$$
To get the reversed inequality, we will use
Theorem~\ref{thm:EMA-infinito}. For this, we write
$\mathcal{A}=\bigcup_{A\in I} S_{E_A}$ which is dense in $S_E$
since $\bigcup_{A\in I} E_A$ is dense in $E$. For every
$A=\{\lambda_1,\ldots,\lambda_m\}\in I$ and every
$a=(a_\lambda)\in S_{E_A}$, we use that $E_A\equiv
\widetilde{E_A}$ to write
$\widetilde{a}=(a_{\lambda_1},\ldots,a_{\lambda_m})\in
S_{\widetilde{E_A}}$ and consider
$\widetilde{b}=(\beta_1,\ldots,\beta_m)\in S_{\widetilde{E_A}^*}$
such that $\sum_{j=1}^m \beta_j a_{\lambda_j} =1$. Now, we define
$(b_\lambda)\in E'$ by the formula
$$
b_\lambda=\beta_j \ \ \text{if $\lambda=\lambda_j$ for $j=1,\ldots,m$} \quad \text{and} \quad
b_\lambda=0\ \ \text{otherwise}.
$$
Then, $(b_\lambda)\in S_{E'}$ and
$$
\sum_{\lambda\in\Lambda} b_\lambda a_\lambda = \sum_{j=1}^m \beta_j a_{\lambda_j} = 1.
$$
This shows that we may use Theorem~\ref{thm:EMA-infinito} to get
\begin{equation*}
n\left(\Bigl[\bigoplus_{\lambda\in \Lambda} X_\lambda\Bigr]_{E}\right) \,
\leq \, \inf\bigl\{n(X_\lambda)\,:\,\lambda\in \Lambda\bigr\}.\qedhere
\end{equation*}
\end{proof}

The hypotheses of the above corollary are clearly satisfied by
$E=c_0(\Lambda)$ and $E=\ell_1(\Lambda)$. For the second space,
the thesis already appeared in our
Corollary~\ref{cor:equality-l1-linfty}. We state the result for
$E=c_0(\Lambda)$. It already appeared in
\cite[Proposition~1]{M-P}.

\begin{corollary}
Let $\Lambda$ be a nonempty set and let $\{X_\lambda\, : \,
\lambda\in\Lambda\}$ be an arbitrary family of Banach spaces. Then
$$
n\left(\Bigl[\bigoplus_{\lambda\in \Lambda}
X_\lambda\Bigr]_{c_0}\right) =
\inf\bigl\{n(X_\lambda)\,:\,\lambda\in \Lambda\bigr\}.
$$
\end{corollary}

Another particular case of Corollary~\ref{cor:equality-E=cupE_A}
appears when we deal with Banach spaces with one-unconditional
basis.

\begin{corollary}
let $E$ be a Banach space with a one-unconditional basis, let
$(j_m)$ be an increasing sequence of positive integers and let
$E_m$ be the subspace of $E$ spanned by the $j_m$ first
coordinates. Suppose $n(E_m)=1$ for every $m\in \N$. Then, for
every sequence $\{X_j\,:\,j\in\N\}$ of Banach spaces,
$$
n\left(\Bigl[\bigoplus_{j\in\N}
X_\lambda\Bigr]_{E}\right)= \inf\{n(X_j)\, :\, j\in \N\}.
$$
\end{corollary}

\section{The numerical index of $L_p$-spaces}\label{sec:Lp}
Our goal in this section is to deduce from the results of the
previous sections that all infinite-dimensional $L_p(\mu)$ spaces
have the same numerical index, independent of the measure $\mu$.
Actually, we will give the result for vector-valued spaces. Let us
say that all the results in this section are already known: they
were proved using particular arguments of the $L_p$ spaces in the
papers \cite{Eddari,Eddari-Khamsi,Eddari-Khamsi-Aksoy} (some of
them with additional unnecessary hypotheses on the measure $\mu$).
In our opinion, the abstract vision we are developing in this
paper allows to understand better the properties of $L_p$-spaces
underlying the proofs: $\ell_p$-sums are absolute sums,
$L_p$-norms are associative, every measure space can be decomposed
into parts of finite measure, every finite measure algebra is
isomorphic to the union of homogeneous measure algebras (Maharam's
theorem) and, finally, the density of simple functions via the
conditional expectation projections.

\begin{prop}\textrm{\cite{Eddari,Eddari-Khamsi,Eddari-Khamsi-Aksoy}}\label{prop:L_p-final}
Let $1<p<\infty$ and let $X$ be any Banach space.
\begin{enumerate}
\item[(a)] The sequence
    $\bigl\{n(\ell_p^m(X))\bigr\}_{m\in\N}$ is decreasing and
$$
n(\ell_p(X)) = \lim_{m\to\infty} n(\ell_p^m(X))\ \Bigl[=\inf_{m\in\N} n(\ell_p^m(X))\Bigr].
$$
\item[(b)] If $\mu$ is a positive measure such that $L_p(\mu)$
    is infinite-dimensional, then
    $$
    n\bigl(L_p(\mu,X)\bigr)=n\bigl(\ell_p(X)\bigr).
    $$
\end{enumerate}
In particular, all infinite-dimensional $L_p(\mu)$ spaces have the
same numerical index.
\end{prop}

\begin{proof}[Proof of Proposition~\ref{prop:L_p-final}.a]
Since $\ell_p^m(X)$ is an absolute summand of $\ell_p^{m+1}(X)$,
the decrease of the sequence
$\bigl\{n(\ell_p^m(X))\bigr\}_{m\in\N}$ follows from
Corollary~\ref{cor:EMA}. By the same reason we have that
$n(\ell_p(X)) \leq n(\ell_p^m(X))$ for every $m\in \N$, and so
$n(\ell_p(X)) \leq \lim_{m\to\infty}  n(\ell_p^m(X))$ (cf.\ also
Corollary~\ref{cor:c0-lp-sum-easyinequality}). For the reversed
inequality, we just use Theorem~\ref{theorem:directedset} with
$Z=\ell_p(X)$, $I=\N$ with its natural order, and
$$
Z_m=\bigl\{x\in \ell_p(X)\ : \ x(k)=0 \ \forall k > m \bigr\}\equiv \ell_p^m(X).
$$
(Observe that not only $Z_m$ is one-complemented in $Z$ but,
actually, the natural projection is absolute. In this case, some
of the arguments in the proof of Theorem~\ref{theorem:directedset}
are simpler.)
\end{proof}

To proof item (b) in Proposition~\ref{prop:L_p-final} we need
the following lemma which is well known to experts, but since
we have not found any concrete reference in the literature, we
include a sketch of its proof for the sake of completeness. We
only need the (trivial) decomposition of every measure space
into disjoint finite measure spaces and Maharam's result on the
decomposition of finite measure algebras into disjoint
homogeneous parts.

\begin{lemma}\label{lemma:Maharam}
Let $(\Omega,\Sigma,\mu)$ be a measure space and $1 < p<\infty$
such that the space $L_p(\mu)$ is infinite-dimensional, and let
$X$ be any non-null Banach space. Then, there is a $\sigma$-finite
measure $\tau$ and a Banach space $Z$ such that $L_p(\tau,X)\neq
0$,
$$
L_p(\mu,X)\equiv L_p(\tau,X) \oplus_p Z \qquad
\text{and} \qquad L_p(\tau,X)\equiv \ell_p\bigl(L_p(\tau,X)\bigr)
\equiv L_p(\tau,\ell_p(X)).
$$
\end{lemma}

\begin{proof}
It is known that for every measure space $(\Omega,\Sigma,\mu)$, there is a
family of finite measure spaces $\{\nu_j\,:\, j\in J\}$ such that
$$
L_p(\mu,X)\equiv \Bigl[\bigoplus_{j\in J} L_p(\nu_j,X)\Bigr]_{\ell_p}.
$$
Indeed, it is enough to consider a maximal family $\{A_j\,:\,j\in
J\}$ of pairwise disjoint sets of positive finite measure and
$\nu_j$ is just the restriction of $\mu$ to $A_j$ (for $p=1$ and
$X=\K$ there is a proof of this fact in \cite[p.~501]{D-F} which
obviously extends to $1<p<\infty$ and arbitrary space $X$). If for
every $j\in J$, $L_p(\nu_j)$ is finite-dimensional, then
$L_p(\mu,X)\equiv \ell_p(\Gamma,X)$ for some infinite set
$\Gamma$, and the result clearly follows by just taking an
infinite and countable subset $\Gamma'$ of $\Gamma$,
$L_p(\tau,X)\equiv \ell_p(\Gamma',X)$ and
$Z=\ell_p(\Gamma\setminus \Gamma',X)$. Otherwise, pick
$\nu=\nu_{j_0}$ to be one of the finite measures $\nu_j$ such that
$L_p(\nu)$ is infinite-dimensional and observe that
$L_p(\mu,X)\equiv L_p(\nu,X)\oplus_p Z_1$ for some subspace $Z_1$
of $L_p(\mu,X)$.

Now, $\nu$ is a finite positive measure such that $L_p(\nu)$ is
infinite-dimensional and we may use Maharam's theorem
\cite[Theorems~7 and 8]{Lac} on measure algebras to deduce (as is
done in \cite[Chapter~5]{Lac} for the scalar case) that
$$
L_p(\nu,X)\equiv \ell_p(\Gamma,X)
\oplus_p \Bigl[\bigoplus_{\alpha \in I} L_p\bigl([0,1]^{\omega_\alpha},X\bigr)\Bigr]_{\ell_p}
$$
for some sets $\Gamma$ and $I$ ($I$ is a set of ordinals). Since
$L_p(\nu,X)$ is infinite-dimensional and $\nu$ is finite, $\Gamma$
is countably infinite or $I$ is nonempty. If $\Gamma$ is infinite,
just consider a measure $\tau$ such that $L_p(\tau,X)\equiv
\ell_p(\Gamma,X)$ and
$$
Z\equiv \Bigl[\bigoplus\limits_{\alpha
\in I} L_p\bigl([0,1]^{\omega_\alpha},X\bigr)\Bigr]_{\ell_p}\oplus_p
Z_1
$$
and the result clearly follows. Otherwise, $I\neq \emptyset$ and
we may take $\alpha_0 \in I$ and consider a measure $\tau$ such
that $L_p(\tau,X) \equiv
L_p\bigl([0,1]^{\omega_{\alpha_0}},X\bigr)$ and
$$
Z \equiv
\ell_p(\Gamma,X) \oplus_p
\Bigl[\bigoplus_{\alpha_0 \neq \alpha \in I}
L_p\bigl([0,1]^{\omega_\alpha},X\bigr)\Bigr]_{\ell_p} \oplus_p
Z_1.
$$
Then, $L_p(\mu,X)\equiv L_p(\tau,X) \oplus_p Z$ and
$$
L_p(\tau,X)\equiv \ell_p\bigl(L_p(\tau,X)\bigr)
$$
by the homogeneity of the measure $\tau$ (see \cite[pp.\ 122 and
128]{Lac}). Finally, it is straightforward to check that
$\ell_p\bigl(L_p(\tau,X)\bigr)\equiv L_p(\tau,\ell_p(X))$ by just
using the associativity of the $\ell_p$-norm.
\end{proof}

\begin{proof}[Proof of Proposition~\ref{prop:L_p-final}.b]
We use Lemma~\ref{lemma:Maharam} to write $L_p(\mu,X)\equiv
L_p(\tau,X)\oplus_p Z$ where the measure $\tau$ is $\sigma$-finite
and satisfies that $L_p(\tau,X)\equiv L_p(\tau,\ell_p(X))$. First,
we deduce from Corollary~\ref{cor:c0-lp-sum-easyinequality} that
$n\bigl(L_p(\mu,X)\bigr)\leq n\bigl(L_p(\tau,X)\bigr)$. On the
other hand, Corollary~\ref{cor:Lp-vector} gives us that
$n\bigl(L_p(\tau,\ell_p(X))\bigr) \leq n\bigl(\ell_p(X)\bigr)$.
Summarizing, we get
$$
n\bigl(L_p(\mu,X)\bigr)\leq n\bigl(L_p(\tau,X)\bigr)=n\bigl(L_p(\tau,\ell_p(X))\bigr)
 \leq n\bigl(\ell_p(X)\bigr).
$$
To get the reversed inequality, we argue in the same way that
we did in the proof of Corollary~\ref{cor:L_1muX-equality}. Let
$I$ be the family of all finite collections of pairwise
disjoint elements of $\Sigma$ with finite measure, ordered by
$\pi_1 \leq \pi_2$ if and only if each element in $\pi_1$ is a
union of elements in $\pi_2$. Then $I$ is a directed set. For
$\pi\in I$, we write $Z_\pi$ for the subspace of $Z=L_p(\mu,X)$
consisting of all simple functions supported in the elements of
$\pi$. Now, for every $\pi\in I$, the subspace $Z_\pi$ is
isometrically isomorphic to $\ell_p^m(X)$-space ($m$ is the
number of elements in $\pi$, see \cite[Lemma~II.2.1]{D-U} for
the case $p=1$ and a finite measure, but the proof is the same
in the general case), it is one-complemented by the conditional
expectation associated to the partition $\pi$ and, finally, the
density of simple functions on $L_p(\mu,X)$ gives that $Z=
\overline{\bigcup_{\pi\in I} Z_\pi}$. Then,
Theorem~\ref{theorem:directedset} applies and it follows that
\begin{equation*}
n\bigl(L_p(\mu,X)\bigr)\geq \limsup_{\pi\in I} n(Z_\pi)
\geq \inf_{m\in \N} n(\ell_p^m(X)) = n(\ell_p(X)).\qedhere
\end{equation*}
\end{proof}

\begin{remark}
Let us comment that the proof of
Proposition~\ref{prop:L_p-final}.b given here can be heavily
simplified for concrete measure spaces for which
Lemma~\ref{lemma:Maharam} can be avoided. For instance, the
fact that $n\bigl(L_p [0,1]\bigr)=n(\ell_p)$ follows
immediately from Corollary~\ref{cor:monotonic-basis} (using the
Haar system as monotone basis of $L_p[0,1]$) and
Corollary~\ref{cor:Lp-vector} (using that
$L_p\bigl([0,1],\ell_p\bigr)\equiv L_p[0,1]$).
\end{remark}

\end{document}